\documentclass[12pt,reqno]{amsart}

\usepackage{amssymb,hyperref}

\newtheorem{theorem}{Theorem}[section]
\newtheorem{proposition}[theorem]{Proposition}
\newtheorem{lemma}[theorem]{Lemma}
\newtheorem{corollary}[theorem]{Corollary}

\numberwithin{equation}{section}

\begin{document}
\baselineskip=15.5pt

\title[Transversely holomorphic branched Cartan geometry]{Transversely holomorphic 
branched Cartan geometry}

\author[I. Biswas]{Indranil Biswas}

\address{School of Mathematics, Tata Institute of Fundamental
Research, Homi Bhabha Road, Mumbai 400005, India}

\email{indranil@math.tifr.res.in}

\author[S. Dumitrescu]{Sorin Dumitrescu}

\address{Universit\'e C\^ote d'Azur, CNRS, LJAD, France}

\email{dumitres@unice.fr}



\date{}

\begin{abstract}
In \cite{BD} we introduced and studied the concept of holomorphic {\it branched 
Cartan geometry}. We define here a foliated version of this notion; this is done in 
terms of Atiyah bundle. We show that any complex compact manifold of algebraic 
dimension $d$ admits, away from a closed analytic subset of positive codimension, a 
nonsingular holomorphic foliation of complex codimension $d$ endowed with a 
transversely flat branched complex projective geometry (equivalently, a ${\mathbb 
C}P^d$-geometry). We also prove that transversely branched holomorphic Cartan 
geometries on compact complex projective rationally connected varieties and on 
compact simply connected Calabi-Yau manifolds are always flat (consequently, they are 
defined by holomorphic maps into homogeneous spaces).
\end{abstract}

\maketitle

\section{Introduction}

In the recent article \cite{BD}, the authors introduced and studied the concept of {\it branched 
Cartan geometry} in the complex setting. This concept generalizes to higher dimension the notion 
of branched (flat) complex projective structure on a Riemann surface introduced and studied by 
Mandelbaum in \cite{M1, M2}. This new framework is much more flexible than that of the usual 
holomorphic Cartan geometries; for example, all compact complex projective manifolds admit 
branched holomorphic projective structures.

In this paper we deal with a foliated version of branched Cartan geometry. More 
precisely, we give a definition, in terms of Atiyah bundle, of a branched 
holomorphic Cartan geometry transverse to a holomorphic foliation. There is a 
natural curvature tensor which vanishes exactly when the transversely branched 
Cartan geometry is flat. When this happens, away from the branching divisor, the 
foliation is transversely modeled on a homogeneous space in the classical sense 
(see, for example, \cite{Mo}). The local coordinates with values in the 
homogeneous space extend through the branching divisor as a ramified holomorphic 
map (the branching divisor correspond to the ramification set). It should be
mentioned transversely holomorphic affine as well as projective structures
for (complex) codimension one foliations are studied extensively (see
\cite{Sc}, \cite{LPT}, \cite{CP} and references therein); such structures are
automatically flat.

In Section \ref{s3}, we use the formalism of Atiyah bundle, to deduce, in the flat 
case, the existence of a developing map which is a holomorphic map $\rho$ from the 
universal cover of the foliated manifold into the homogeneous space; the
differential $d \rho$ of $\rho$
surjective on an open dense set of the universal cover, and the foliation
on it is given by the kernel of $d 
\rho$. We also show that any complex compact manifold of algebraic dimension $d$ 
admits, away from a closed analytic subset of positive codimension, a nonsingular 
holomorphic foliation of complex codimension $d$, endowed with a transversely flat 
branched complex projective geometry (which is same as a ${\mathbb C}P^d$-geometry).

In Section \ref{s4} we use characteristic classes to prove a criterion for a 
holomorphic foliation to admit a branched transversely Cartan geometry. In 
particular, the criterion asserts that, on compact K\"ahler manifolds, foliations 
$\mathcal F$ with strictly negative conormal bundle do not admit any branched 
transversely holomorphic Cartan geometry whose model is the complex affine space 
(which is same as a holomorphic affine connections).

In Section \ref{special varieties} we consider holomorphic foliations $\mathcal F$ on two classes 
of special manifolds $\widehat X$: projective rationally connected manifolds, and simply connected 
Calabi-Yau manifolds. In both cases, we show that all transversely branched holomorphic Cartan 
geometries (on the open dense set $X$ of $\widehat X$ where the foliation is nonsingular) are 
necessarily flat and come from a holomorphic map into a homogeneous space with surjective 
differential at the general point.

\section{Foliation and transversely branched Cartan geometry} 

\subsection{Partial connection along a foliation}

Let $X$ be a connected complex manifold equipped with a
nonsingular holomorphic foliation $\mathcal F$;
so, $\mathcal F$ is a holomorphic subbundle of the holomorphic tangent bundle $TX$ such that the
sheaf of holomorphic sections of $\mathcal F$ is closed under the Lie bracket operation of vector
fields. Let
$$
{\mathcal N}_{\mathcal F}\, :=\, TX/{\mathcal F}\, \longrightarrow\, X
$$
be the normal bundle to the foliation. Let
\begin{equation}\label{q}
q\, :\, TX\, \longrightarrow\, N_{\mathcal F}
\end{equation}
be the quotient map. There is a natural flat holomorphic partial connection 
${\nabla}^{\mathcal F}$ on $N_{\mathcal F}$ in the direction of $\mathcal F$. We will briefly 
recall the construction of ${\nabla}^{\mathcal F}$. Given locally defined holomorphic 
sections $s$ and $t$ of $\mathcal F$ and ${\mathcal N}_{\mathcal F}$ respectively, choose a locally 
defined holomorphic section $\widetilde t$ of $TX$ that projects to $t$. Now define
$$
{\nabla}^{\mathcal F}_s t\,=\, q([s,\, {\widetilde t}])\, ,
$$
where $q$ is the projection in \eqref{q}. it is easy to see that this is 
independent of the choice of the lift $\widetilde t$ of $t$. Indeed, if $\widehat{t}$
is another lift of $t$, then $[s,\, {\widetilde t}-\widehat{t}]$ is a section of
${\mathcal F}$, because ${\widetilde t}-\widehat{t}$ is a section of
${\mathcal F}$. From the Jacobi 
identity for Lie bracket it follows that the curvature of ${\nabla}^{\mathcal F}$ 
vanishes identically.

We will define partial connections in a more general context.

Let $H$ be a complex Lie group. Its Lie algebra will be denoted by $\mathfrak h$. Let
\begin{equation}\label{g1}
p\, :\, E_H\,\longrightarrow\, X
\end{equation}
be a holomorphic principal $H$--bundle on $X$. This means that $E_H$ is a complex manifold
equipped with a holomorphic action
$$
p'\, :\, E_H\times H\, \longrightarrow\, E_H
$$
of $H$, and $p$ is a holomorphic surjective submersion, such that
\begin{itemize}
\item $p\circ p'\,=\, p\circ p_E$, where $p_E\, :\, E_H\times H\, \longrightarrow\, E_H$
is the natural projection, and

\item the map $p_E\times p'\, :\, E_H\times H\, \longrightarrow\, E_H\times_X E_H$ is an
isomorphism; note that the first condition ensures that the image of $p_E\times p'$ is
contained in $E_H\times_X E_H\, \subset\, E_H\times E_H$.
\end{itemize}

Let
\begin{equation}\label{dp}
\mathrm{d}p\, :\, TE_H\, \longrightarrow\, p^*TX
\end{equation}
be the differential of the map $p$ in \eqref{g1}. This homomorphism $\mathrm{d}p$ is surjective
because $p$ is a submersion. The kernel of $\mathrm{d}p$ is identified with the trivial vector
bundle $E_H\times {\mathfrak h}$ using the action of $H$ on $E_H$ (equivalently, by the
Maurer--Cartan form).
Consider the action of $H$ on $TE_H$ given by the action of $H$ on $E_H$.
It preserves the sub-bundle $\text{kernel}(\mathrm{d}p)$.
Define the quotient
$$
\text{ad}(E_H)\, :=\, \text{kernel}(\mathrm{d}p)/H\, \longrightarrow\, X\, .
$$
This $\text{ad}(E_H)$ is a holomorphic vector bundle over $X$. In fact, it is identified
with the vector bundle $E_H\times^H\mathfrak h$ associated to $E_H$ for the adjoint
action of $H$ on $\mathfrak h$; this identification is given by the
above identification of $\text{kernel}(\mathrm{d}p)$ with $E_H\times {\mathfrak h}$. This vector
bundle $\text{ad}(E_H)$ is known as the adjoint vector bundle for $E_H$. Since the
adjoint action of $H$ on $\mathfrak h$ preserves its Lie algebra structure, for any
$x\, \in\, X$, the fiber $\text{ad}(E_H)_x$ is a Lie algebra isomorphic to $\mathfrak h$. In
fact, $\text{ad}(E_H)_x$ is identified with $\mathfrak h$ uniquely up to a conjugation.

The direct image $p_*TE_H$ is equipped with an action of $H$ given by the action
of $H$ on $TE_H$. Note that $p_*TE_H$ is a locally free quasi-coherent analytic sheaf
on $X$. Its $H$--invariant part
$$
(p_*TE_H)^H \,\subset\, p_*TE_H
$$
is a locally free coherent analytic sheaf on $X$. The corresponding holomorphic vector bundle is denoted
by $\text{At}(E_H)$; it is known as the Atiyah bundle for $E_H$ \cite{At}. It is straight-forward
check that the quotient
$$
(TE_H)/H\, \longrightarrow\, X
$$
is identified with $\text{At}(E_H)$. Consider the short exact sequence of holomorphic vector
bundles on $E_H$
$$
0\, \longrightarrow\, \text{kernel}(\mathrm{d}p)\, \longrightarrow\,
\mathrm{T}E_H \, \stackrel{\mathrm{d}p}{\longrightarrow}\,p^*TX
\, \longrightarrow\, 0\, .
$$
Taking its quotient by $H$, we get the following short exact sequence of
vector bundles on $X$
\begin{equation}\label{at1}
0\, \longrightarrow\, \text{ad}(E_H)\, \stackrel{\iota''}{\longrightarrow}\,\text{At}(E_H)\,
\stackrel{\widehat{\mathrm{d}}p}{\longrightarrow}\, TX\, \longrightarrow\, 0\, ,
\end{equation}
where $\widehat{\mathrm{d}}p$ is constructed from $\mathrm{d}p$; this is known as
the Atiyah exact sequence for $E_H$. Now define the subbundle
\begin{equation}\label{atF}
\text{At}_{\mathcal F}(E_H)\, :=\, (\widehat{\mathrm{d}}p)^{-1}({\mathcal F})\,\subset\,
\text{At}(E_H)\, .
\end{equation}
So from \eqref{at1} we get the short exact sequence
\begin{equation}\label{at2}
0\, \longrightarrow\, \text{ad}(E_H)\, \longrightarrow\,\text{At}_{\mathcal F}(E_H)\,
\stackrel{\mathrm{d}'p}{\longrightarrow}\, {\mathcal F}\, \longrightarrow\, 0\, ,
\end{equation}
where $\mathrm{d}'p$ is the restriction of $\widehat{\mathrm{d}}p$
in \eqref{at1} to the subbundle $\text{At}_{\mathcal F}(E_H)$.

A partial holomorphic connection on $E_H$ in the direction of $\mathcal F$ is a holomorphic
homomorphism
$$
\theta\, :\, {\mathcal F}\, \longrightarrow\, \text{At}_{\mathcal F}(E_H)
$$
such that $\mathrm{d}'p\circ\theta\,=\, \text{Id}_{\mathcal F}$, where
$\mathrm{d}'p$ is the homomorphism in \eqref{at2}. Giving such a homomorphism
$\theta$ is equivalent to giving a homomorphism
$\varpi\, :\, \text{At}_{\mathcal F}(E_H)\, \longrightarrow\, \text{ad}(E_H)$ such that the
composition
$$
\text{ad}(E_H) \,\hookrightarrow\,
\text{At}_{\mathcal F}(E_H)\, \stackrel{\varpi}{\longrightarrow}\, \text{ad}(E_H)
$$
is the identity map of $\text{ad}(E_H)$, where the inclusion of $\text{ad}(E_H)$ in
$\text{At}_{\mathcal F}(E_H)$ is the injective homomorphism in \eqref{at2}. Indeed, the
homomorphisms $\varpi$ and $\theta$ uniquely determine each other by the condition that
the image of $\theta$ is the kernel of $\varpi$.

Given a partial connection $\theta\, :\, {\mathcal F}\, \longrightarrow\, \text{At}_{\mathcal F}(E_H)$,
and any two locally defined holomorphic sections $s_1$ and $s_2$ of $\mathcal F$, consider
the locally defined section $\varpi ([\theta(s_1),\, \theta(s_2)])$ of $\text{ad}(E_H)$ (since
$\theta(s_1)$ and $\theta(s_2)$ are $H$--invariant vector fields on $E_H$, the Lie bracket
$[\theta(s_1),\, \theta(s_2)]$ is also an $H$--invariant vector field). This defines an
${\mathcal O}_X$--linear homomorphism
$$
{\mathcal K}(\theta) \, \in\, H^0(X,\, \text{Hom}(\bigwedge\nolimits^2{\mathcal F},\, \text{ad}(E_H)))
\,=\, H^0(X,\, \text{ad}(E_H)\otimes \bigwedge\nolimits^2{\mathcal F}^*)\, ,
$$
which is called the \textit{curvature} of the connection $\theta$. The connection $\theta$ is
called flat if ${\mathcal K}(\theta)$ vanishes identically.

A partial connection on $E_H$ induces a partial connection on every bundle 
associated to $E_H$. In particular, a partial connection on $E_H$ induces a 
partial connection on the adjoint bundle $\text{ad}(E_H)$.

Since ${\rm At}_{\mathcal F}(E_H)$ is a subbundle of ${\rm At}(E_H)$, any partial connection
$\theta\, :\, {\mathcal F}\, \longrightarrow\, \text{At}_{\mathcal F}(E_H)$ produces a homomorphism
${\mathcal F}\, \longrightarrow\, \text{At}(E_H)$; this homomorphism will be denoted by
$\theta'$. Note that from \eqref{at1} we have an exact sequence
\begin{equation}\label{at3}
0\, \longrightarrow\, \text{ad}(E_H)\, \stackrel{\iota'}{\longrightarrow}\,
\text{At}(E_H)/\theta'({\mathcal F})\,
\stackrel{\widehat{\mathrm{d}}p}{\longrightarrow}\, TX/{\mathcal F}\,
=\, {\mathcal N}_{\mathcal F}\, \longrightarrow\, 0\, ,
\end{equation}
where $\iota'$ is given by $\iota''$ in \eqref{at1}.

\begin{lemma}\label{lem1}
Let $\theta$ be a flat partial connection on $E_H$. Then $\theta$ produces a flat partial
connection on ${\rm At}(E_H)/\theta'({\mathcal F})$ that satisfies the condition that the
homomorphisms in the exact sequence \eqref{at3} are connection preserving.
\end{lemma}

\begin{proof}
The image of $\theta$ defines an $H$--invariant holomorphic foliation on $E_H$; 
let
\begin{equation}\label{wtf}
{\widetilde F}\, \subset\, TE_H
\end{equation}
be this foliation. Note that the differential
$\mathrm{d}p$ in \eqref{dp} produces an isomorphism of ${\widetilde F}$ with
$p^*\mathcal F$. The natural connection on 
the normal bundle $TE_H/{\widetilde F}$ in the direction of ${\widetilde F}$ is 
evidently $H$--invariant (recall that $\text{At}(E_H)\,=\, (TE_H)/H$). On the other
hand, we have $(TE_H/{\widetilde F})/H\,=\, {\rm 
At}(E_H)/\theta'({\mathcal F})$. Therefore, the above connection on 
$TE_H/{\widetilde F}$ in the direction of ${\widetilde F}$ descends to a flat 
partial connection on ${\rm At}(E_H)/\theta'({\mathcal F})$ in the direction on 
$\mathcal F$.

Let $s$ be a holomorphic section of $\mathcal F$ defined on an open subset
$U\, \subset\, X$. Let $s'$ be the unique section of ${\widetilde F}$ over
$p^{-1}(U)\, \subset\, E_H$ such that $\mathrm{d}p(s')\,=\, s$. Let $t$ be a
holomorphic section of $\text{kernel}(\mathrm{d}p)\, \subset\, TE_H$ over
$p^{-1}(U)$. Then the Lie bracket $[s', \,t]$ has the property that 
$\mathrm{d}p([s', \,t])\,=\, 0$, meaning $[s', \,t]$ is a section of
$\text{kernel}(\mathrm{d}p)$. Since $\text{ad}(E_H)\,=\, \text{kernel}(\mathrm{d}p)
/H$, it now follows that the inclusion of $\text{ad}(E_H)$ in 
${\rm At}(E_H)/\theta'({\mathcal F})$ in \eqref{at3} preserves the partial connections
on $\text{ad}(E_H)$ and ${\rm At}(E_H)/\theta'({\mathcal F})$ 
in the direction of $\mathcal F$. Since $[s', \,t]$ is a section of
$\text{kernel}(\mathrm{d}p)$, it also follows that the projection
$\widehat{\mathrm{d}}p$ in \eqref{at3} is also partial connection preserving.
\end{proof}

\subsection{Transversely branched Cartan geometry}

Let $G$ be a connected complex Lie group and $H\, \subset\, G$ a complex Lie subgroup.
The Lie algebra of $G$ will be denoted by $\mathfrak g$. As in \eqref{g1}, $E_H$
is a holomorphic principal $H$--bundle on $X$. Let
\begin{equation}\label{eg}
E_G\,=\, E_H\times^H G\,\longrightarrow\, X
\end{equation}
be the principal $G$--bundle on $X$ obtained by extending the structure group of $E_H$ using the 
inclusion of $H$ in $G$. The inclusion of $\mathfrak h$ in $\mathfrak g$ produces a fiber-wise
injective homomorphism of Lie algebras
\begin{equation}\label{i1}
\iota\, :\, \text{ad}(E_H)\,\longrightarrow\,\text{ad}(E_G)\, ,
\end{equation}
where $\text{ad}(E_G)\,=\, E_G\times^G{\mathfrak g}$ is the adjoint bundle for $E_G$.
Let $\theta$ be a flat partial connection on $E_H$ in the direction of 
$\mathcal F$. So $\theta$ induces flat partial connections on the associated bundles $E_G$, 
$\text{ad}(E_H)$ and $\text{ad}(E_G)$.

A transversely branched holomorphic Cartan geometry of type $(G,\, H)$
on the foliated manifold $(X,\, {\mathcal F})$ is
\begin{itemize}
\item a holomorphic principal $H$--bundle $E_H$ on $X$ equipped with a flat partial
connection $\theta$, and

\item a holomorphic homomorphism
\begin{equation}\label{beta}
\beta\,:\, \text{At}(E_H)/\theta'({\mathcal F})\, \longrightarrow\,
\text{ad}(E_G)\, ,
\end{equation}
\end{itemize}
such that the following three conditions hold:
\begin{enumerate}
\item $\beta$ is partial connection preserving,

\item $\beta$ is an isomorphism over a nonempty open subset of $X$, and

\item the following diagram is commutative:
\begin{equation}\label{cg1}
\begin{matrix}
0 &\longrightarrow & \text{ad}(E_H) &\stackrel{\iota'}{\longrightarrow} &
\text{At}(E_H)/\theta'({\mathcal F}) &
\longrightarrow & {\mathcal N}_{\mathcal F} &\longrightarrow & 0\\
&& \Vert &&~ \Big\downarrow\beta && ~ \Big\downarrow\overline{\beta}\\
0 &\longrightarrow & \text{ad}(E_H) &\stackrel{\iota}{\longrightarrow} &
\text{ad}(E_G) &\longrightarrow &
\text{ad}(E_G)/\text{ad}(E_H) &\longrightarrow & 0
\end{matrix}
\end{equation}
\end{enumerate}
where the top exact sequence is the one in \eqref{at3}, and $\iota$ is the homomorphism
in \eqref{i1}.

From the commutativity of \eqref{cg1} it follows immediately that the homomorphism 
$\overline{\beta} \,:\, {\mathcal N}_{\mathcal F}\,\longrightarrow\, 
\text{ad}(E_G)/\text{ad}(E_H)$ in \eqref{cg1} is an isomorphism 
over a point $x\, \in\, X$ if and only if $\beta(x)$ is an isomorphism.

Let $n$ be the complex dimension of $\mathfrak g$. Consider the homomorphism of $n$-th
exterior products
$$
\bigwedge\nolimits^n\beta\, :\, \bigwedge\nolimits^n(\text{At}(E_H)/\theta'({\mathcal F}))
\, \longrightarrow\, \bigwedge\nolimits^n\text{ad}(E_G)
$$
induced by $\beta$. The homomorphism $\beta$ fails to be an isomorphism precisely over the
divisor of the section $\bigwedge\nolimits^n\beta$ of the line bundle
$\text{Hom}(\bigwedge\nolimits^n(\text{At}(E_H)/\theta'({\mathcal F})),
\,\bigwedge\nolimits^n\text{ad}(E_G))$. This divisor $\text{div}(\bigwedge\nolimits^n\beta)$
will be called the \textit{branching divisor} for $((E_H,\, \theta),\, \beta)$.
We will call $((E_H,\, \theta),\, \beta)$ a holomorphic
Cartan geometry if $\beta$ is an isomorphism over $X$.

Take a holomorphic principal $H$--bundle $E_H$ on $X$ equipped with a flat partial
connection $\theta$ in the direction of
$\mathcal F$. Giving a homomorphism $\beta$ as in \eqref{beta} satisfying the
above conditions is equivalent to giving a holomorphic $\mathfrak g$--valued one--form
$\omega$ on $E_H$ satisfying the following conditions:
\begin{enumerate}
\item $\omega$ is $H$--equivariant for the adjoint action of $H$ on $\mathfrak g$,

\item $\omega$ vanishes on the foliation ${\widetilde F}\, \subset\, TE_H$ in
\eqref{wtf} given by the image of $\theta$,

\item the resulting homomorphism $\omega\, :\, (TE_H)/{\widetilde F}
\, \longrightarrow\, E_H\times
{\mathfrak g}$ is an isomorphism over a nonempty open subset of $E_H$, and

\item the restriction of $\omega$ to any fiber of $p$ (see \eqref{g1}) coincides with the
Maurer--Cartan form for the action of $H$ on the fiber.
\end{enumerate}

To see that the two descriptions of a transversely branched holomorphic Cartan 
geometry are equivalent, first recall
that $p^* \text{At}(E_H)\,=\, TE_H$, and the pullback of $p^*\text{ad}(E_G)$ is
identified with the trivial vector bundle $E_H\times{\mathfrak g}\, \longrightarrow\,
E_H$. Given a homomorphism $\beta\,:\, \text{At}(E_H)/\theta'({\mathcal F})\,
\longrightarrow\, \text{ad}(E_G)$ satisfying the above conditions, the composition
$$
TE_H\,=\, p^* \text{At}(E_H)\,\longrightarrow\,p^*(\text{At}(E_H)/\theta'({\mathcal F}))
\,\stackrel{p^*\beta}{\longrightarrow}\, p^*\text{ad}(E_G)\,=\,
E_H\times {\mathfrak g}
$$
defines a holomorphic $\mathfrak g$--valued one--form
$\omega$ on $E_H$ that satisfies the above conditions.
Conversely, any holomorphic $\mathfrak g$--valued one--form
$\omega$ on $E_H$ that satisfying the above conditions, produces a homomorphism
$$
(TE_H)/{\widetilde F} \, \longrightarrow\, E_H\times\mathfrak g
$$
because it vanishes on $\widetilde F$. This homomorphism is $H$--equivariant, so
descends to a homomorphism
$$
\text{At}(E_H)/\theta'({\mathcal F})\,=\, 
((TE_H)/{\widetilde F})/H\, \longrightarrow\, (E_H\times\mathfrak g)/H\,=\,
\text{ad}(E_G)
$$
over $X$. This descended homomorphism satisfies the conditions needed to
define a transversely branched holomorphic Cartan
geometry.

If $\mathcal F$ is the trivial foliation (by points) then the previous definition 
is exactly that of a branched Cartan geometry on $X$, as given in \cite{BD}.

\section{Connection and developing map}\label{s3}

\subsection{Holomorphic connection on $E_G$}

Let $((E_H,\, \theta),\, \beta)$ be a transversely branched Cartan geometry of type 
$(G,\, H)$ on the foliated manifold $(X,\, {\mathcal F})$. We will show that this 
data produces a holomorphic connection on the principal $G$--bundle $E_G$ defined in 
\eqref{eg}.

Consider the homomorphism
\begin{equation}\label{eh}
\text{ad}(E_H)\,\longrightarrow\, {\rm ad}(E_G)
\oplus \text{At}(E_H)\, , \ \ v \, \longmapsto\,
(\iota(v),\, -\iota''(v))
\end{equation}
(see \eqref{i1} and \eqref{at1} for $\iota$ and $\iota''$ respectively).
The corresponding quotient $({\rm ad}(E_G)\oplus \text{At}(E_H))/\text{ad}(E_H)$
is identified with the Atiyah bundle
${\rm At}(E_G)$. The inclusion of $\text{ad}(E_G)$ in ${\rm At}(E_G)$ as in
\eqref{at1} is given by the inclusion $\text{ad}(E_G) \, \hookrightarrow\,
{\rm ad}(E_G)\oplus \text{At}(E_H)$, $w\, \longmapsto\, (w,\, 0)$, while the projection
${\rm At}(E_G)\, \longrightarrow\, TX$ is given by the
composition
$$
{\rm At}(E_G)\, \hookrightarrow\, {\rm ad}(E_G)
\oplus \text{At}(E_H) \,\stackrel{(0,\widehat{\mathrm{d}}p)}{\longrightarrow}\, TX\, ,
$$
where $\widehat{\mathrm{d}}p$ is the projection in \eqref{at1}.

Consider the subbundle $\theta'({\mathcal F})\, \subset\, \text{At}(E_H)$ in \eqref{at3}.
The composition
$$
{\rm At}(E_H)\, \longrightarrow\, \text{At}(E_H)/\theta'({\mathcal F})
\,\stackrel{\beta}{\longrightarrow}\, \text{ad}(E_G)\, ,
$$
where the first homomorphism is the quotient map, will be denoted by $\beta'$. The homomorphism
\begin{equation}\label{hv}
{\rm ad}(E_G)\oplus \text{At}(E_H)\, \longrightarrow\, {\rm ad}(E_G)\, , \ \
(v,\, w) \, \longmapsto\, v+\beta'(w)
\end{equation}
vanishes on the image of $\text{ad}(E_H)$ by the map in \eqref{eh}. Therefore,
the homomorphism in \eqref{hv} produces a homomorphism
\begin{equation}\label{vp}
\varphi\, :\, \text{At}(E_G)\,=\, ({\rm ad}(E_G)\oplus \text{At}(E_H))/\text{ad}(E_H)
\,\longrightarrow\, \text{ad}(E_G)\, .
\end{equation}
The composition
$$
\text{ad}(E_G)\,\hookrightarrow\, \text{At}(E_G)\, \stackrel{\varphi}{\longrightarrow}\,\text{ad}(E_G)
$$
clearly coincides with the identity map of $\text{ad}(E_G)$. Hence $\varphi$ defines a holomorphic
connection on the principal $G$--bundle $E_G$ \cite{At}. Note that $\theta$ is not
directly used in the construction of the homomorphism $\varphi$. Let
$$
{\rm Curv}(\varphi)\, \in\, H^0(X,\, \text{ad}(E_G)\otimes\Omega^2_X)
$$
be the curvature of the connection $\varphi$.

\begin{lemma}\label{lem2}
The curvature ${\rm Curv}(\varphi)$ lies in the image of the homomorphism
$$
H^0(X,\, {\rm ad}(E_G)\otimes\bigwedge\nolimits^2 {\mathcal N}^*_{\mathcal F})
\, \hookrightarrow\, H^0(X,\, {\rm ad}(E_G)\otimes\Omega^2_X)
$$
given by the inclusion $q^*\, :\, {\mathcal N}^*_{\mathcal F}\, \hookrightarrow\,
\Omega^1_X$ (the dual of the projection in \eqref{q}).
\end{lemma}

\begin{proof}
Let $\widetilde\theta$ be the partial connection on $E_G$ induced by the partial 
connection $\theta$ on $E_H$. Note that $\widetilde\theta$ is flat because 
$\theta$ is flat. Since the homomorphism $\beta$ in \eqref{beta} is partial connection 
preserving, it follows that the restriction of the connection $\varphi$ in the 
direction of $\mathcal F$ coincides with $\widetilde\theta$. Hence the restriction 
of $\varphi$ to $\mathcal F$ is flat.

In fact, since $\beta$ is connection preserving, the contraction of ${\rm Curv}( 
\varphi)$ by any tangent vector of $TX$ lying in $\mathcal F$ vanishes. This 
implies that ${\rm Curv}(\varphi)$ is actually a section of ${\rm ad}(E_G)\otimes 
\bigwedge\nolimits^2 {\mathcal N}^*_{\mathcal F}$.
\end{proof}

The transversely branched Cartan geometry $((E_H,\, \theta),\, \beta)$ will be 
called \textit{flat} if the curvature ${\rm Curv}(\varphi)$ vanishes identically.

\subsection{The developing map}\label{developing}

Assume that $((E_H,\, \theta),\, \beta)$ is flat and $X$ is simply connected. Fix a point
$x_0\,\in\, X$ and a point $z_0\in\, (E_H)_{x_0}$ in the fiber of $E_H$ over $x_0$. Using the
flat connection $\varphi$ on $E_G$ and the trivialization of $(E_G)_{x_0}$ given by $z_0$,
the principal $G$--bundle $E_G$ gets identified with $X\times G$. Using this identification, the
inclusion of $E_H$ in $E_G$ produces a holomorphic map
\begin{equation}\label{rho}
\rho\, :\, X\, \longrightarrow \, G/H\, .
\end{equation}
If the base point $z_0$ is replaced by $z_0h\,\in\, (E_H)_{x_0}$, where $h\,\in\, H$, then
the map $\rho$ in \eqref{rho} gets replaced by the composition
$$
X\, \stackrel{\rho}{\longrightarrow} \, G/H\, \stackrel{y\mapsto hy}{\longrightarrow} \,G/H\, .
$$
The map $\rho$ will be called a developing map for $((E_H,\, \theta),\, \beta)$.

The differential of $\rho$ is surjective outside the branching divisor for $((E_H,\, 
\theta),\, \beta)$. Indeed, the differential $d\rho\, :\, TX\, \longrightarrow \, 
\rho^* T(G/H)$ of $\rho$ is given by the homomorphism $\overline{\beta}$ in 
\eqref{cg1}. It was noted earlier that $\overline{\beta}$ fails to be an 
isomorphism exactly over the branching divisor for $((E_H,\, \theta),\, \beta)$.

Note that $\rho$ is a constant map when restricted to a connected component of a 
leaf for $\mathcal F$, because the connection $\varphi$ restricted to such a 
connected component is induced by a connection on $E_H$ (it is induced by the 
partial connection $\theta$ on $E_H$). In particular, $\rho$ is a constant map if 
there is a dense leaf for $\mathcal F$. In that case, ${\rm rank}({\mathcal 
N}_{\mathcal F}) \,=\, \dim \mathfrak g - \dim \mathfrak h \,=\, 0$, so $X$ the
unique leaf.

If $X$ is not simply connected, fix a base point $x_0\, \in\, X$, and let $\psi\, 
:\, \widetilde{X}\, \longrightarrow\, X$ be the corresponding universal cover. 
Considers the pull-back $\widetilde{\mathcal F}$ of the foliation $\mathcal F$, as 
well as the pull-back of the transversely branched flat Cartan geometry $((E_H,\, 
\theta),\, \beta)$, to $\widetilde{X}$ using $\psi$. Then the developing map of 
the transversely flat Cartan geometry on $(\widetilde{X},\, \widetilde{\mathcal 
F})$ is a holomorphic map $\rho \,:\, \widetilde{X}\,\longrightarrow\, G/H$ (as 
before, we need to fix a point in $(\psi^*E_H)_{x'_0}$, where $x'_0\, \in\, 
\widetilde{X}$ is the base point), which is a submersion away from the inverse 
image, under $\psi$, of the branching divisor). Moreover, the monodromy of the 
flat connection on $E_G$ produces a group homomorphism (called monodromy 
homomorphism) from the fundamental group $\pi_1(X, x_0)$ of $X$ into $G$, and 
$\rho$ must be equivariant with respect to the action of $\pi_1(X, x_0)$ by 
deck-transformation on $\widetilde{X}$ and through the image of the monodromy 
morphism on $G/H$. The reader will find more details about this construction in 
\cite{Mo}.

\subsection{Fibrations over a homogeneous space}\label{fibration hom} 

The standard (flat) Cartan geometry on the homogeneous space $X=G/H$ is given by 
the following tautological construction.

Let $F_H$ be the holomorphic principal $H$--bundle on $X$ defined by the quotient 
map $G\, \longrightarrow\, G/H$ (we use the notation $F_H$ instead of $E_H$ 
because it is a special case which will play a role later). Identify the Lie 
algebra $\mathfrak g$ with the Lie algebra of right--invariant vector fields on 
$G$. This produces an isomorphism
\begin{equation}\label{bgh}
\beta_{G,H}\, :\, \text{At}(F_H)\, \longrightarrow\, \text{ad}(F_G)
\end{equation}
and hence a Cartan geometry of type $G/H$ on $X$ (the foliation on $G/H$ (leaves
are points) is trivial and there is no branching divisor).

The principal $G$--bundle
$$F_G\, :=\, F_H\times^H G\, \longrightarrow\, X\,=\, G/H\, ,$$ obtained by extending the
structure group of $E_H$ using the inclusion of $H$ in $G$, is canonically identified with
the trivial principal $G$--bundle $X\times G$. To see this, consider the map
\begin{equation}\label{mapp}
G\times G\, \longrightarrow\, G\times G\, , \ \ (g_1,\, g_2)\, \longmapsto\, (g_1,\, g_1g_2)\, .
\end{equation}
Note that $E_G$ is the quotient of $G\times G$ where any $(g_1h,\, g_2)$ is identified with
$(g_1,\, g_2)$, where $g_1,\, g_2\,\in\, G$ and $h\, \in\, H$. Therefore, the map in
\eqref{mapp} produces an isomorphism of $E_G$ with $X\times G$. The connection on $F_G$
given by the above Cartan geometry of type $G/H$ on $X\,=\, G/H$ is the trivial connection on
$X\times G$. In particular, the Cartan geometry of type $G/H$ on $X$ is flat.

The above holomorphic $\mathfrak g$--valued $1$--form on $G\,=\, F_H$ will be 
denoted by $\beta_{G,H}$.

Let $X$ be a connected complex manifold and
$$
\gamma\, :\, X\, \longrightarrow\, G/H
$$
a holomorphic map such that the differential
$$
d\gamma\, :\, TX\, \longrightarrow\, T(G/H)
$$
is surjective over a nonempty subset of $X$.

Consider the foliation on $X$ given by the kernel of $d \gamma$. It is a singular 
holomorphic foliation, which is regular on the dense open set of $X$ where the 
homomorphism $d \gamma$ is surjective. It extends to a regular holomorphic 
foliation
\begin{equation}\label{exfo}
{\mathcal F}\, \subset\, TX'
\end{equation}
on an open subset $X'$ of $X$ of complex codimension at
least two (containing the open set where $d \gamma$ is surjective).

Set $E_H$ to be the pullback $\gamma^*F_H$. 

Note that we have a holomorphic
map $\eta\, :\, E_H\, \longrightarrow\, F_H$ which is $H$--equivariant and
fits in the commutative diagram
$$
\begin{matrix}
E_H & \stackrel{\eta}{\longrightarrow} & F_H\\
\Big\downarrow && \Big\downarrow\\
X & \stackrel{\gamma}{\longrightarrow} & G/H
\end{matrix}
$$

Notice that, by construction, the $H$-bundle $E_H$ is trivial along the leaves of 
$\mathcal F$ and hence it inherits a flat partial connection $\theta$ along the 
leafs of the foliation $\mathcal F$ constructed in \eqref{exfo}. Let
\begin{equation}\label{thp}
\theta'\, :\, {\mathcal F}\, \longrightarrow\, \text{At}(E_H)\,=\, \text{At}(\gamma^*F_H)
\end{equation}
be the homomorphism giving this partial connection.

We will show that $(E_H,\, \eta^*\beta_{G,H})$ defines a transversely holomorphic 
branched flat Cartan geometry of type $G/H$ on the foliated manifold $(X',\, 
{\mathcal F})$, where $X'\, \subset\, X$ is the dense open subset introduced 
earlier. It is branched over points $x \,\in\, X'$ where $d \gamma (x)$ is not 
surjective.

To describe the above branched Cartan geometry in terms of the Atiyah bundle, first note
that $\text{At}(E_H)=\text{At}(\gamma^* F_H)$ coincides with the subbundle of the vector bundle
$\gamma^*\text{At}(F_H) \oplus TX$ given by the kernel of the homomorphism
$$
\gamma^*\text{At}(F_H)\oplus TX\, \longrightarrow\, \gamma^*T(G/H)\, ,\ \
(v,\, w)\, \longmapsto\, \gamma^*p_{G,H}(v') -d\gamma(w)\, ,
$$
where $p_{G,H}\, :\, \text{At}(F_H)\, \longrightarrow\, T(G/H)$ is the natural projection
(see \eqref{at1}), while $v'$ is the image of $v$ under the natural map
$\gamma^*\text{At}(F_H)\,\longrightarrow\, \text{At}(F_H)$,
and $$d\gamma\,:\, TX\,\longrightarrow\, \gamma^* T(G/H)$$ is the
differential of $\gamma$. 

Notice that the restriction of the homomorphism
$$
\gamma^*\text{At}(F_H)\oplus TX\, \longrightarrow\, \gamma^*\text{ad}(F_G)\, ,
\ \ (a,\, b)\, \longmapsto\, \gamma^*\beta_{G,H} (a)
$$
(see \eqref{bgh} for $\beta_{G,H}$)
to $\text{At}(\gamma^*F_H)\, \subset\, \gamma^*\text{At}(F_H)\oplus TX$ is a homomorphism
$$
\text{At}(\gamma^*F_H)\, \longrightarrow\, \text{ad}(\gamma^*F_G)\,=\, \gamma^*\text{ad}(F_G)
\,=\, \text{ad}(E_G)\, ,
$$
which vanishes on $\theta'(\mathcal F)$, where $\theta'$ is constructed in \eqref{thp}.

It defines a transversely branched holomorphic Cartan geometry of type $G/H$ on 
$(X', \mathcal F)$.

The divisor of $X'$ over which the above branched transversely Cartan geometry of type $G/H$ on 
$X'$ fails to be a Cartan geometry coincides with the divisor over which the 
differential $d\gamma$ fails to be surjective.

It was observed earlier that the model Cartan geometry defined by $\beta_{G,H}$ in \eqref{bgh} 
is flat. As a consequence of it, the above branched transversely Cartan geometry of type $G/H$ 
on $X'$ is flat.

The developing map for this flat branched Cartan geometry on $X'$, is the map 
$\gamma$ itself restricted to $X'$.

The following proposition is proved similarly.

\begin{proposition}\label{fibration}
Let $X$ be connected complex manifold, and let $M$ be a complex manifold endowed with a 
holomorphic Cartan geometry of type $(G,H)$. Suppose that there exists a holomorphic map $f\,:\, 
X \,\longrightarrow\, M$ such that the differential $df$ is surjective on an open dense subset of
$X$. Then the kernel of $df$ defines a holomorphic foliation $\mathcal F$ on an open dense subset $X'$ 
of $X$ of complex codimension at least two. Moreover $\mathcal F$ admits a transversely branched 
holomorphic Cartan geometry of type $(G,H)$, which is flat if and only if the Cartan geometry on 
$M$ is flat.
\end{proposition}

\begin{proof} 
The proof is the same as above if one considers, instead of $\gamma^*F_H$ and 
$\gamma^*\beta_{G,H}$, the pull back of the Cartan geometry of $M$ through $f$. 
\end{proof}

\subsection{Transversely affine and transversely projective geometry}

Let us recall two standard models $G/H$ which are of particular interest: the 
complex affine and the complex projective geometries.

Consider the semi-direct product ${\mathbb C}^d\rtimes\text{GL}(d, {\mathbb C})$ for the standard 
action of $\text{GL}(d, {\mathbb C})$ on ${\mathbb C}^d$. This group ${\mathbb 
C}^d\rtimes\text{GL}(d, {\mathbb C})$ is identified with the group of all affine transformations 
of ${\mathbb C}^d$. Set $H\,=\, \text{GL}(d, {\mathbb C})$ and $G\,=\, {\mathbb 
C}^d\rtimes\text{GL}(d, {\mathbb C})$.

By definition, a given regular holomorphic foliation $\mathcal F$ of complex codimension 
$d$ admits a transversely (branched) {\it holomorphic affine connection} if it 
admits a transversely (branched) holomorphic Cartan geometry of type $G/H$. When 
the transversely Cartan geometry is flat, we say that $\mathcal F$ admits a 
transversely (branched) {\it complex affine geometry}.

We also recall that a holomorphic foliation $\mathcal F$ of complex codimension 
$d$ admits a transversely (branched) {\it holomorphic projective connection} if it 
admits a (branched) holomorphic Cartan geometry of type $\text{PGL}(d+1,{\mathbb 
C})/Q$, where $Q\, \subset\, \text{PGL}(d+1,{\mathbb C})$ is the maximal parabolic 
subgroup that fixes a given point for the standard action of 
$\text{PGL}(d+1,{\mathbb C})$ on ${\mathbb C}P^d$ (the space of lines in ${\mathbb 
C}^{d+1}$). If the transversely Cartan geometry is flat, we say that $\mathcal F$ 
admits a transversely (branched) {\it complex projective geometry}.

We have seen in Section \ref{fibration hom} that any holomorphic map $X 
\,\longrightarrow\, {\mathbb C}P^d$ which is a submersion on an open dense set 
gives rise to a holomorphic foliation with transversely branched complex 
projective geometry. Conversely, we have seen in Section \ref{developing} that on 
simply connected manifolds, any foliation with transversely branched complex 
projective geometry is given by a holomorphic map $X \,\longrightarrow\, {\mathbb 
C}P^d$ which is a submersion on an open dense set.

Consider now a complex manifold $X$ of algebraic dimension $a(X)=d$. Recall that 
the algebraic dimension is the degree of transcendence over $\mathbb C$ of the 
field ${\mathcal M}(X)$ of meromorphic functions on $X$. It is known that $a(X)$ 
is at most the complex dimension of $X$ with equality if and only if $X$ is 
birational to a complex projective manifold (see \cite{Ue}), known as Moishezon
manifolds.

\begin{proposition}
Suppose that $X$ is a compact complex manifold of algebraic dimension $a(X)\,=\,d$. 
Then, away from an analytic subset of positive codimension, $X$ admits a 
nonsingular holomorphic foliation of complex codimension $d$, endowed with a 
transversely branched complex projective geometry.
\end{proposition}

\begin{proof}
This is a direct application of the algebraic reduction theorem (see \cite{Ue}) which asserts that 
$X$ admits a modification $\widehat{X}$ such that there exists a holomorphic surjective map $f 
\,:\, \widehat{X}\,\longrightarrow\, Y$ to a compact complex projective manifold $Y$ of complex 
dimension $d$ such that $$f^* \,:\, {\mathcal M} (Y)\,\longrightarrow\,{\mathcal 
M}(\widehat{X})\,=\, {\mathcal M}(X)$$ is an isomorphism. Moreover, since $Y$ is projective,
there exists a finite algebraic map $\pi \,:\, Y \,\longrightarrow\, {\mathbb C}P^d$ (see a 
short proof of this classical fact in \cite{BD} Proposition 3.1). Hence we get a holomorphic 
surjective fibration $\pi \circ f \,:\, \widehat{X} \,\longrightarrow\, {\mathbb C}P^d$. Now 
Proposition \ref{fibration} applies.
\end{proof}

\section{A topological obstruction}\label{s4}

Let $X$ be a compact connected K\"ahler manifold of complex dimension $d$ equipped with a K\"ahler
form $\omega$. Chern classes will always mean ones with real coefficients. For a torsionfree
coherent analytic sheaf $V$ on $X$, define
\begin{equation}\label{deg}
\text{degree}(V)\,:=\, (c_1(V)\cup\omega^{d-1})\cap [X]\, \in\, {\mathbb R}\, .
\end{equation}
The degree of a divisor $D$ on $X$ is defined to be $\text{degree}({\mathcal O}_X(D))$.

Fix an effective divisor $D$ on $X$. Fix a holomorphic principal $H$--bundle $E_H$ 
on $X$.

\begin{proposition}\label{thm1}
Let $\mathcal F$ be a holomorphic nonsingular foliation on the 
K\"ahler manifold $X$. Assume that $\mathcal F$ admits a transversely branched Cartan geometry 
of type $G/H$ with principal $H$--bundle $E_H$ and branching divisor $D$. Then ${\rm degree}({\mathcal 
N}^*_{\mathcal F})-{\rm degree}(D)\, =\, {\rm degree}({\rm ad}(E_H))$.

In particular, if $D\, \not=\, 0$, then ${\rm degree}({\mathcal N}^*_{\mathcal F})\,>\, 
{\rm degree}({\rm ad}(E_H))$.
\end{proposition}

\begin{proof} Let $k$ be the complex dimension of the transverse model geometry $G/H$.

Recall that the homomorphism $\overline{\beta} \,:\, {\mathcal N}_{\mathcal 
F}\,\longrightarrow\, \text{ad}(E_G)/\text{ad}(E_H)$ in \eqref{cg1} is an 
isomorphism over a point $x\, \in\, X$ if and only if $\beta(x)$ is an isomorphism.

The branching divisor $D$ coincides with the vanishing divisor of the holomorphic section $\bigwedge^k \overline{ \beta}$ of the holomorphic line bundle $\bigwedge^k ({\mathcal N}^*_{\mathcal F}) \otimes \bigwedge^k ({\rm ad}(E_G)/{\rm ad}(E_H))$. We have
$$
\text{degree}(D)\,=\, \text{degree}(\bigwedge\nolimits^k ({\rm ad}(E_G)/{\rm ad}(E_H))
\otimes \bigwedge^k ({\mathcal N}^*_{\mathcal F}) )
$$
\begin{equation}\label{f2}
=\, \text{degree}({\rm ad}(E_G)) - \text{degree}({\rm ad}(E_H))
+ {\rm degree}({\mathcal N}^*_{\mathcal F})\, .
\end{equation}
Recall that $E_G$ has a holomorphic connection $\phi$ (see \eqref{vp}).
It induces a holomorphic connection on $\text{ad}(E_G)$. Hence we have
$c_1({\rm ad}(E_G)) \,=\, 0$ \cite[Theorem~4]{At}, which implies that
$\text{degree}({\rm ad}(E_G)) \,=\, 0$. Therefore, from \eqref{f2} it follows that
\begin{equation}\label{e7}
{\rm degree}({\mathcal N}^*_{\mathcal F})-{\rm degree}(D)\, =\, {\rm degree}({\rm ad}(E_H))\, .
\end{equation}

If $D\,\not=\, 0$, then $\text{degree}(D)\, >\, 0$. Hence in that case \eqref{e7} yields
${\rm degree}({\mathcal N}^*_{\mathcal F} )\, >\, {\rm degree}({\rm ad}(E_H))$.
\end{proof}

\begin{corollary}\label{corollaire deg}\mbox{}
\begin{enumerate}
\item[(i)] If ${\rm degree}({\mathcal N}^*_{\mathcal F})\, <\, 0$, then there is no
branched transversely holomorphic affine connection on $X$ transversal to ${\mathcal F}$.

\item[(ii)] If ${\rm degree}({\mathcal N}^*_{\mathcal F})\, =\, 0$, then for every branched
transversely holomorphic affine connection on $X$ transversal to $\mathcal F$
the branching divisor on $X$ is trivial.
\end{enumerate}
\end{corollary}

\begin{proof}
Recall that a transversely branched holomorphic affine connection on $X$ transversal
to $\mathcal F$ is a transversely branched holomorphic Cartan geometry on $X$ of type $G/H$,
where $H\,=\,
\text{GL}(d, {\mathbb C})$ and $G\,=\, {\mathbb C}^d\rtimes\text{GL}(d, {\mathbb C})$.
The homomorphism $$\text{M}(d, {\mathbb C})\otimes
\text{M}(d, {\mathbb C})\, \longrightarrow\, \mathbb C\, ,\ \ A\otimes B\, \longmapsto\,
\text{trace}(AB)$$ is nondegenerate and $\text{GL}(d, {\mathbb C})$--invariant. In other words, 
the Lie algebra $\mathfrak h$ of $H\,=\, \text{GL}(d, {\mathbb C})$
is self-dual as an $H$--module. Hence we have $\text{ad}(E_H)\,=\, \text{ad}(E_H)^*$, in particular,
the equality $$\text{degree}(\text{ad}(E_H))\,=\,0$$ holds. Hence from
Proposition \ref{thm1},
\begin{equation}\label{prc}
{\rm degree}({\mathcal N}^*_{\mathcal F})\,=\, {\rm degree}(D)\, .
\end{equation}

As noted before, for a nonzero effective divisor $D$ we have $\text{degree}(D)\, >\, 0$.
Therefore, the corollary follows from \eqref{prc}.
\end{proof}

\section{Flatness of the transverse geometry on some special varieties}\label{special varieties}

In this section we consider holomorphic foliations $\mathcal F$ on projective 
rationally connected manifolds and on simply connected Calabi-Yau manifolds 
$\widehat X$. In both cases, we show that the only transversely branched 
holomorphic Cartan geometries, on the open dense set $X$ of $\widehat X$ where the 
foliation is nonsingular, are necessarily flat and come from a holomorphic map 
into a homogeneous space (as described in Section \ref{fibration hom}).

\subsection{Rationally connected varieties}

Let $\widehat X$ be a smooth complex projective rationally connected variety. Let $X\, \subset\,
\widehat{X}$ be a Zariski open subset such that the complex codimension of the complement
$\widehat{X}\setminus X$ is at least two. Take a nonsingular foliation
$$
{\mathcal F}\, \subset\, TX
$$
on $X$. Let $((E_H,\, \theta),\, \beta)$ be a transversely branched holomorphic Cartan geometry of
type $(G,\, H)$ on the foliated manifold $(X,\, {\mathcal F})$.

There is a nonempty open subset of $X$ which can be covered by smooth complete 
rational curves $C$ such that the restriction $(TX)\vert_C$ is ample. On a curve 
any holomorphic connection is flat. Further, since a rational curve is simply 
connected, any holomorphic bundle on it equipped with a holomorphic connection is 
isomorphic to the trivial bundle equipped with the trivial connection. If 
$(TX)\vert_C$ is ample, then $H^0(C,\, (\Omega^2_X)\vert_C) \,=\, 0$. Therefore, 
any holomorphic bundle on $X$ with a holomorphic connection has the property that 
the curvature vanishes identically. In particular, the transversely branched 
holomorphic Cartan geometry $((E_H,\, \theta),\, \beta)$ must be flat. Further, $X$ 
is simply connected because $\widehat X$ is so. Therefore, the transversely 
branched holomorphic Cartan geometry $((E_H,\, \theta),\, \beta)$ is the pullback, 
of the standard Cartan geometry on $G/H$ of type $(G,\, H)$, by a developing map 
$f\, :\, X\, \longrightarrow\, G/H$. The foliation is given by 
$\text{kernel}(df)$.

This yields the following:

\begin{corollary}\label{corflat}
Let $\widehat X$ be a smooth complex projective 
rationally connected variety, and let $\mathcal F$ be a holomorphic nonsingular 
foliation of positive codimension defined on a Zariski open subset $X$ of complex 
codimension at least two in $\widehat{X}$. Then there is no transversely branched 
Cartan geometry, with model a nontrivial analytic affine variety $G/H$, on $X$ 
transversal to $\mathcal F$. In particular, there is transversely holomorphic 
affine connection on $X$ transversal to $\mathcal F$.
\end{corollary}

\begin{proof}
Assume, by contradiction, that there is a transversely 
branched Cartan geometry on $X$ whose model is an analytic affine variety $G/H$. By
the above observations, the branched Cartan geometry is necessarily flat and is
given by a holomorphic developing map $f\, :\, X\, \longrightarrow\, G/H$. Since the
target $G/H$ is an affine analytic variety, Hartog's theorem says that $f$ extends to a 
holomorphic map $\widehat{f}\, :\, \widehat{X}\, \longrightarrow\, G/H$. Now, as 
$\widehat{X}$ is compact, and $G/H$ is affine, $\widehat{f}$ must be constant: a 
contradiction; indeed, as for $f$, the differential of $\widehat{f}$ is injective
on $TX/{\mathcal F}$ at a general point of $X$.
\end{proof}

Notice that if $G$ is a complex linear algebraic group and $H$ a closed reductive 
algebraic subgroup, then $G/H$ is an affine analytic variety (see Lemma 3.32 in 
\cite{Mc}).

\subsection{Simply connected Calabi--Yau manifolds}

Let $\widehat X$ be a simply connected compact K\"ahler manifold with $c_1(\widehat{X})\,=\, 0$. 
As before, $X\, \subset\, \widehat{X}$ is a dense open subset such that the complement 
$\widehat{X}\setminus X$ is a complex analytic subset of complex codimension at least two.
Take a nonsingular foliation
$$
{\mathcal F}\, \subset\, TX
$$
on $X$. Take a complex Lie group $G$ such that there is a holomorphic homomorphism
$G\, \longrightarrow\, \text{GL}(n,{\mathbb C})$ with the property that the corresponding
homomorphism of Lie algebras is injective.

Let $((E_H,\, \theta),\, \beta)$ be a transversely branched holomorphic Cartan geometry of type 
$(G,\, H)$ on the foliated manifold $(X,\, {\mathcal F})$. Consider the holomorphic connection 
$\varphi$ on $E_G$ over $X$ (see \eqref{vp}). The principal $G$--bundle $E_G$ extends to a 
holomorphic principal $G$--bundle $\widehat{E}_G$ over $\widehat X$, and the connection 
$\varphi$ extends to a holomorphic connection $\widehat\varphi$ on $\widehat{E}_G$ \cite[Theorem 
1.1]{Bi}. We know that $\widehat{E}_G$ is the trivial holomorphic principal $G$--bundle, and 
$\widehat\varphi$ is the trivial connection \cite[Theorem 6.2]{BD}. Also, $X$ is simply 
connected because $\widehat X$ is so. Therefore, the transversely branched holomorphic Cartan 
geometry $((E_H,\, \theta),\, \beta)$ is the pullback, of the standard Cartan geometry on $G/H$ of 
type $(G,\, H)$, by a developing map $f\, :\, X\, \longrightarrow\, G/H$. The foliation is given 
by $\text{kernel}(df)$.

As before we have the following:

\begin{corollary}
Let $\widehat X$ be a simply connected Calabi-Yau manifold,
and let $\mathcal F$ be a holomorphic nonsingular
foliation of positive codimension defined on a Zariski open subset $X$ of complex
codimension at least two in $\widehat{X}$. Then there is no transversely branched
Cartan geometry, with model a nontrivial analytic affine variety $G/H$, on $X$
transversal to $\mathcal F$. In particular, there is transversely holomorphic
affine connection on $X$ transversal to $\mathcal F$.
\end{corollary}

Its proof is identical to that of Corollary \ref{corflat}.

%%%%%%%%%%%%%%%%%%%%%%%%%%%%%%%%%%%%%%%%%%%%%%%%%%%%%%%%%%%%%%%%%

\end{document}